\documentclass[10pt,a4paper,oneside,reqno]{amsart}
\usepackage[utf8x]{inputenc}
\usepackage{ucs}
\usepackage{amsthm}
\usepackage{amsmath}
\usepackage{amsfonts}
\usepackage{amssymb}
\usepackage{url}
\newtheorem{Theorem}{Theorem}

\newtheorem{Proposition}[Theorem]{Proposition}

\newtheorem{Lemma}[Theorem]{Lemma}

\author{Aled Walker}
\title{A multiplicative analogue of Schnirelmann's theorem}

\begin{document}

\begin{abstract}
The classical theorem of Schnirelmann states that the primes are an additive basis for the integers. In this paper we consider the analogous multiplicative setting of the cyclic group $\left(\mathbb{Z}/ q\mathbb{Z}\right)^{\times}$, and prove a similar result. For all suitably large primes $q$ we define $P_\eta$ to be the set of primes less than $\eta q$, viewed naturally as a subset of $\left(\mathbb{Z}/ q\mathbb{Z}\right)^{\times}$. Considering the $k$-fold product set $P_\eta^{(k)}=\{p_1p_2\cdots p_k:p_i\in P_\eta \}$, we show that for $\eta \gg q^{-\frac{1}{4}+\epsilon}$ there exists a constant $k$ depending only on $\epsilon$ such that $P_\eta^{(k)}=\left(\mathbb{Z}/ q\mathbb{Z}\right)^{\times}$. Erd\H{o}s conjectured that for $\eta = 1$ the value $k=2$ should suffice: although we have not been able to prove this conjecture, we do establish that $P_1 ^{(2)}$ has density at least $\frac{1}{64}(1+o(1))$. We also formulate a similar theorem in almost-primes, improving on existing results. 
\end{abstract}
\maketitle
\section{Main Theorems}
For any abelian group $G$ written multiplicatively, and a subset $S\subseteq G$, we define the \emph{k-fold iterated product-set} $S^{(k)}$ to be the set $\{s_1s_2\cdots s_k:s_i\in S \}$. Letting $q$ be a large prime\footnote{The methods used to prove Theorems \ref{positive density theorem} and \ref{almost prime theorem} may be modified to the case $q$ composite, although the technical details become increasingly complicated. However, it is not immediately apparent that the rather oblique arguments used in Lemma \ref{coset proposition} admit such a modification. For simplicity we restrict to $q$ prime throughout.}, we are interested in when certain naturally defined subsets $S\subseteq\left(\mathbb{Z}/ q\mathbb{Z}\right)^{\times}$ generate the entire group, in the additive-combinatorial sense that $S^{(k)}= \left(\mathbb{Z}/ q\mathbb{Z}\right)^{\times}$ for some $k$. With certain sets $S$, most notably intervals, the case $k=2$ has been extensively studied, and is known as the \emph{modular hyperbola} problem -- see the survey of Shparlinski \cite{Sh12}. One can also consider the modular hyperbola problem for sets of primes, and in \cite{E-O-S87}, defining $P_1$ to be the set of all primes less than $q$,  Erd\H{o}s conjectured that $P_1^{(2)}=\left(\mathbb{Z}/ q\mathbb{Z}\right)^{\times}$. This conjecture is still open, even assuming the Generalised Riemann Hypothesis\footnote{However, on GRH, an easy Fourier-analytic method shows that $P_1^{(3)}=(\mathbb{Z}/q\mathbb{Z})^\times$.}. In this paper we show various unconditional partial results in the direction of Erd\H{o}s' conjecture, improving upon some existing work. \\

\textbf{Acknowledgements.} The author is very grateful to Adam Harper for suggesting a simplification to the original proof of Theorem \ref{positive density theorem} and to Igor Shparlinski for notifying him of the work in \cite{Sh13}. The author\footnote{Contact email address: walker@maths.ox.ac.uk} is a DPhil student at Oxford University, supported by EPSRC Grant MATH1415, and is much indebted to supervisor Ben Green for his continual support and encouragement. \\

Let us fix notation. As usual the letter $p$ will always denote a prime, and for a large prime $q$ and $\eta\leqslant 1$ we define $P_\eta=\{p:p<\eta q\}$. We reserve $q$ for this fixed large prime. \\

We now state the main results of the paper.
\begin{Theorem}
\label{positive density theorem}
Let $\epsilon>0$ and take $\eta = q^{-\frac{1}{4}+\epsilon}$. Then there exist constants $q_0(\epsilon)$ and $c(\epsilon)$ such that for $q\geqslant q_0(\epsilon)$ we have $\lvert P_\eta^{(2)} \rvert \geqslant c(\epsilon) q$. We calculate we may take $c(\epsilon)=(\frac{2\epsilon}{3+4\epsilon})^2(1+o(1))$, and so in particular may take $c(\frac{1}{4})=\frac{1}{64}(1+o(1))$. 
\end{Theorem}

The proof of this theorem employs sieve weights to upper-bound the number of solutions to $p_1p_2\equiv a(\operatorname{mod} q)$ for a fixed $a$, from which we conclude that the support of $P_\eta^{(2)} $ cannot be too small. The key feature, which we believe to be relatively novel, is that while we lose information by switching to sieve weights we also gain by accessing stronger $L_1$ bounds of their Fourier transform.

We establish another partial result by solving the modular hyperbola problem in almost-primes.

\begin{Theorem}
\label{almost prime theorem}
Let $\epsilon>0$.
\begin{enumerate}
\item For large enough $q$, every non-zero residue modulo $q$ can be expressed as the product of at most $6$ primes less than $q$. In fact, there exists $q_0(\epsilon)$ such that, for $q\geqslant q_0(\epsilon)$, every non-zero residue modulo $q$ can be expressed as the product of at most $6$ primes less than $q^{\frac{15}{16}+\epsilon}$. 
\item There exists $q_0(\epsilon)$ and $k(\epsilon)\in \mathbb{N}$ such that, for $q\geqslant q_0(\epsilon)$, every non-zero residue modulo $q$ can be expressed as the product of at most $k(\epsilon)$ primes less than $q^{\frac{3}{4}+\epsilon}$. 
\end{enumerate}
\end{Theorem}

\noindent Note that $\frac{15}{16}= 0.9375$. This theorem improves a result of \cite{Sh13}, in which it is established that all such residues can be expressed as the product of at most $18$ primes less than $q^{0.997}$. 

Finally, we deduce that every residue may be expressed as the product of a fixed number of small primes.

\begin{Theorem}
\label{exact theorem}
Let $\epsilon>0$.
\begin{enumerate}
\item Take $\eta = q^{-\frac{1}{16}+\epsilon}$. Then there exists $q_0(\epsilon)$ such that for $q\geqslant q_0(\epsilon)$ we have $P_{\eta}^{(48)}=(\mathbb{Z}/q\mathbb{Z})^\times$.
\item Take $\eta = q^{-\frac{1}{4}+\epsilon}$. Then there exists $q_0(\epsilon)$ and $K(\epsilon)\in\mathbb{N}$ such that for $q\geqslant q_0(\epsilon)$ we have $P_{\eta}^{(K(\epsilon))}=(\mathbb{Z}/q\mathbb{Z})^\times$.
\end{enumerate}
\end{Theorem}

\noindent The second part of Theorems \ref{almost prime theorem} and \ref{exact theorem} may be viewed as multiplicative analogies of the classical Schnirelmann's theorem that every sufficiently large integer is the sum of at most 37000 primes, proved in \cite{Sch33} (see exposition in \cite{Na96}). There is analogy too between the methods of proof: we use Theorem \ref{positive density theorem} to establish a positive density result, and then an argument from additive combinatorics to show that this dense set expands. 

Consider $q=5$: we see $P_1=\{2,3\}$ consists entirely of quadratic non-residues, and $P_1^{(2)}=\{1,4\}$, $P_1^{(3)}=\{2,3\}$, $P_1^{(4)}=\{1,4\}$ etcetera, and so Theorem \ref{exact theorem} fails to hold. The obstruction arises as $P$ is entirely contained within a coset of a non-trivial subgroup $H\leqslant \left(\mathbb{Z}/p\mathbb{Z}\right)^{\times}$, or equivalently has a non-trivial Fourier coefficient of maximal value. In Lemma \ref{coset proposition} we establish that, for large enough $q$, the primes less than $q^{\frac{1}{4}+\epsilon}$ cannot be trapped in such a coset. Unfortunately we have not been able to improve upon the weak indirect argument used there, and hence have not been able to show any genuine cancellation in the Fourier coefficients. In Section \ref{Proof of Theorem 3} we discuss why finding an improved result may be difficult. \\

We end this introduction by surveying other partial results towards Erd\H{o}s' conjecture, in addition to \cite{Sh13}. The original paper \cite{E-O-S87} shows that, under the Generalised Riemann Hypothesis, there are at most $c\log^5 q$ residues $a<q$ that may not be expressed as the product of two primes less than $q$. The authors of \cite{FKS08} average over $q$ in a certain range, establishing unconditionally a similar result for almost-all $q$. Corollary 1 of a preprint\footnote{This preprint has been withdrawn, owing to an error in a different part of the authors' argument. However, the cited result remains sound -- it is available at the referenced url.} of Heath-Brown and Li \cite{HB-Li15} implies unconditionally that for almost-all $q$ we have $P_\eta^{(3)}=(\mathbb{Z}/q\mathbb{Z})^\times$, with $\eta=q^{-\frac{1}{2}+\epsilon}$. 

\section{Facts from sieve theory}
The proofs of Theorems \ref{positive density theorem} and \ref{almost prime theorem} will be applications of certain sieve weights. In this section we collect together the precise results required, and discuss suitable references. 
\begin{Proposition}[Upper-bound sieve]
\label{upper bound sieve weights}
Let $\gamma>0$, $\xi$ be a fixed real satisfying $0<\xi<\frac{1}{2}-\frac{\gamma}{2}$, and $x$ be a large integer, i.e. $x\geqslant x_0(\gamma)$ for some $x_0(\gamma)$. Let $z=x^\xi$ and $D=x^{2\xi}$. We denote by $\nu(d)$ the number of distinct prime factors of $d$. Then there exists a weight-function $w^+:[x]\longrightarrow \mathbb{R}_{\geqslant 0}$ such that:
\begin{enumerate}
\item if $n$ has no prime factors less than $z$ then $w^+(n)\geqslant 1$
\item $\sum\limits_{n=1}^{x}w^+(n)\leqslant (1+o_{x\rightarrow\infty}(1))\frac{x}{\xi\log x}$
\item $w^+(n)=\sum\limits_{d\vert n}\lambda^+_d$, where $(\lambda^+_d)_{d\geqslant 1}$ is a sequence of reals satisfying $\lambda^+_d = 0$ for $d>D$, $\lambda^+_d = 0$ for $d$ not square-free, and $\vert \lambda^+_d\rvert \leqslant 3^{\nu(d)}$.  
\end{enumerate}
\end{Proposition}

\begin{proof}
The standard Selberg sieve weights suffice, e.g. the construction giving Theorem 7.1 of \cite{F-I10} with $g(p)\equiv \frac{1}{p}$ for all primes, employing the asymptotic expression for $J(D)$ which begins page 118 of the same volume.
\end{proof}
\begin{Proposition}[Lower-bound sieve]
\label{lower bound sieve weights}
Let $\gamma,\delta>0$, $\xi$ be a fixed real satisfying $0<\xi<\frac{1}{2}-\frac{\gamma}{2}-\frac{\delta}{2}$, and $x$ be a large integer, i.e. $x\geqslant x_0(\gamma,\delta)$ for some $x_0(\gamma,\delta)$. Let $z=x^\xi$ and $D=x^{2\xi+\delta}$. Then there exists a weight-function $w^-:[x]\longrightarrow \mathbb{R}$ such that:
\begin{enumerate}

\item if $n$ has no prime factors less than $z$ then $w^-(n)\leqslant 1$
\item if $n$ has some prime factor that is less than $z$, then $w^-(n)\leqslant 0$
\item there exists a positive $c(\delta)$ such that $\sum\limits_{n=1}^{x}w^-(n)\geqslant (1+o_{x\rightarrow\infty}(1))c(\delta)\frac{x}{\xi\log x}$ 
\item $w^-(n)=\sum\limits_{d\vert n}\lambda^-_d$, where $(\lambda^-_d)_{d\geqslant 1}$ is a sequence of reals satisfying $\lambda^-_d = 0$ for $d>D$ and $\vert \lambda^-_d\rvert \leqslant 1$.  
\end{enumerate}
\end{Proposition}

\begin{proof}
The weights $\lambda^-_d$ are constructed by applying the optimal linear sieve to the sequence $\mathcal{A}=[x]$, with the required results proved in Chapter 11 of \cite{F-I10} and summarised at the beginning of Chapter 12. When sieving the integers $\mathcal{A}=[x]$ with any lower-bound combinatorial linear sieve $(\lambda^-_d)$, with sieving level $z=x^\xi$ and level of support $D= x^{2\xi +\delta}$, by construction the weight $w^-(n)=\sum\limits_{d\vert n}\lambda^-_d$ immediately satisfies parts (i), (ii) and (iv) of the above theorem. To establish part (iii), we note that the right-hand-side of equation (12.13) of \cite{F-I10} is exactly an estimation of the quantity $\sum\limits_{n=1}^{x}w^-(n)$ with optimised weights. In those authors' notation we have $s=\frac{2\xi +\delta}{\xi}$, which is at least $2+2\delta$: therefore $f(s)>0$ and the main term of (12.13) is of the order required in part (iii). Since $\mathcal{A}=[x]$ the errors $\lvert r_d (\mathcal{A})\rvert$ are $O(1)$, and since $D\leqslant x^{1-\gamma}$ the error $R(\mathcal{A},D)$ is negligible compared to the main term -- the proposition is proved. Another useful reference for the linear sieve is Chapter 8 of \cite{Ha-Ri74}, in which Theorem 8.4 may also be used for this proof. 
\end{proof}

In the sequel we shall only use the properties of these weights stated in Propositions \ref{upper bound sieve weights} and \ref{lower bound sieve weights}. Once $x$ is fixed, we shall freely consider these weights as functions on $\mathbb{N}$, supported on $[x]$. \\

To finish this section, let us develop two results on the Fourier theory of sieve weights. We recall the usual definitions, if only to fix normalisations. For an arbitrary function $f:\mathbb{Z}/q\mathbb{Z}\longrightarrow \mathbb{C}$ and $r\in \mathbb{Z}/q\mathbb{Z}$, identifying $[q]$ and $\mathbb{Z}/q\mathbb{Z}$ in a harmless manner, we define the additive Fourier coefficient $$\widehat{f}(r)=\sum\limits_{a\in \mathbb{Z}/q\mathbb{Z}} f(a)e\left(-\frac{ra}{q}\right).$$ Taking $\chi:(\mathbb{Z}/q\mathbb{Z})^\times \longrightarrow \mathbb{C}$ to be a multiplicative character, we define the multiplicative Fourier coefficient $$\widehat{f}(\chi)=\sum\limits_{x\in (\mathbb{Z}/q\mathbb{Z})^\times}f(x)\overline{\chi(x)}.$$ In section \ref{Proofs of first two main theorems} we will also need the usual notion of multiplicative convolution. For two functions $f,g:(\mathbb{Z}/q\mathbb{Z})^\times \longrightarrow \mathbb{C}$ we define their convolution $f\ast g: (\mathbb{Z}/q\mathbb{Z})^\times \longrightarrow \mathbb{C}$ to be the function $$(f\ast g)(a)=\sum\limits_{\substack{x,y\in (\mathbb{Z}/q\mathbb{Z})^\times \\ xy=a}} f(x)\overline{g(y)}.$$\\

Sieve weights, being weighted sums of arithmetic progressions, enjoy cancellation in their non-trivial Fourier coefficients, both additive and multiplicative. The following two lemmas formalise this notion; they hold for the weights coming from either of the two previous propositions but, for ease of application, we state them only for the weights to which they will be applied. 

\begin{Lemma}
Let $w^+$ be as in Proposition \ref{upper bound sieve weights}, with $x< q$. Then 
\begin{equation}
\label{l1 hat norm bound}
\sum\limits_{r=1}^{q-1} \vert \widehat{w^+}(r)\vert \ll qx^{2\xi+o(1)}\log q
\end{equation}
\end{Lemma}
\begin{proof}
The left-hand-side of (\ref{l1 hat norm bound}) may be written explicitly as $$\sum\limits_{r=1}^{q-1}\left\vert \sum\limits_{n=1}^{x}\sum\limits_{d\vert n}\lambda^+_d e\left(-\frac{rn}{q}\right) \right\vert. $$ Swapping the summation over $d$ and $n$, and using the pointwise bound $3^{\nu(d)}\ll d^{o(1)}$, the above expression is at most 

\begin{equation}
\label{after sum swap}
 x^{o(1)}\sum\limits_{r=1}^{q-1}\sum\limits_{d\leqslant x^{2\xi}}\left\vert\sum\limits_{y\leqslant \frac{x}{d}}e\left(-\frac{rdy}{q}\right)\right\vert.
\end{equation}
 
We denote the inner sum by $S$. By the standard estimate $$\left\vert \sum\limits_{x \leqslant X}e\left(\frac{ax}{q}\right)\right\vert\ll \operatorname{max}\left(\frac{q}{a},\frac{q}{q-a}\right)$$ for any $a$ in the range $1\leqslant a\leqslant q-1$, we conclude that   

\begin{equation}
\vert S\vert\ll \operatorname{max}\left(\frac{q}{rd \operatorname{mod} q},\frac{q}{q - rd \operatorname{mod} q}\right)
\end{equation}

\noindent where $rd \operatorname{mod} q$ is the least positive residue congruent to $rd$, noting that $rd$ is not a multiple of $q$. Substituting this bound into (\ref{after sum swap}) yields  $$\sum\limits_{r=1}^{q-1} \vert \widehat{w^+}(r)\vert \ll x^{o(1)}\sum\limits_{r=1}^{q-1}\sum\limits_{d\leqslant x^{2\xi}}\operatorname{max}\left(\frac{q}{rd \operatorname{mod} q},\frac{q}{q - rd \operatorname{mod} q}\right).$$

Swapping the sums over $r$ and $d$, we see that for each fixed $d$ the value $rd \operatorname{mod} q$ achieves each value from $1$ to $q-1$ exactly once. Splitting the sum into those $r$ for which $rd \operatorname{mod} q$ is less than $\frac{q}{2}$, and those for which $rd \operatorname{mod} q$ is greater than $\frac{q}{2}$, we obtain the lemma.
\end{proof}

We now use the Polya-Vinogradov theorem to bound the non-trivial multiplicative Fourier coefficients of sieve weights with small-support. This short argument was suggested to us by Adam Harper. 

\begin{Lemma}
\label{Fourier coefficient prop}
Let $w^-$ be the weight from Proposition \ref{lower bound sieve weights}, and $x\leqslant q$. Then for every non-trivial character $\chi$ we have the bound $$\vert\widehat{w^-}(\chi)\vert\ll x^{2\xi+\delta}q^{\frac{1}{2}}\log q$$
\end{Lemma}

\begin{proof}
For $\chi$ a non-trivial character we have

\begin{align}
\vert\widehat{w^-}(\chi)\vert &= \lvert\sum\limits_{n\leqslant x}w^-(n)\overline{\chi}(n)\rvert \nonumber\\
&\leqslant \sum\limits_{d\leqslant x^{2\xi+\delta}}\vert \lambda_d \vert\lvert \sum\limits_{\substack{n\leqslant x\\d \vert n}}\overline{\chi}(n)\rvert\nonumber\\ &\leqslant x^{2\xi+\delta}q^{\frac{1}{2}}\log q\nonumber
\end{align}

\noindent with the final line following from the Polya-Vinogradov theorem (see Chapter 23 of \cite{Da00}).  
\end{proof}

\section{Proof of Theorems \ref{positive density theorem} and \ref{almost prime theorem}}
\label{Proofs of first two main theorems}
We will now use the results of the previous section, together with Weil's bound for Kloosterman sums, to prove the first two main theorems. The author originally presented a long intricate proof of Theorem \ref{positive density theorem}, in the case $\eta = 1$, in which a three-dimensional small sieve was applied to upper bound the multiplicative energy\footnote{The multiplicative energy is the number of solutions to $p_1p_2 \equiv p_3p_4 (\operatorname{mod} q)$ with $p_i<q$ for all $i$.} of the primes less than $q$. Kloosterman sum bounds were used to obtain the required level of distribution. The author observed that the energy bound was equivalent to bounding the fourth-moment $$\sum\limits_{\chi}\left\vert\sum\limits_{p<q}\chi(p)\right\vert^4$$ and, applying Lemma \ref{Fourier coefficient prop} with an upper-bound sieve weight, Adam Harper noted that this formulation admitted an easy proof. In particular, the Kloosterman sum estimation was no longer required. However, the constant $c(\frac{1}{4})$ in Theorem \ref{positive density theorem} given by this method is rather small -- around $\frac{1}{4000}$. To bring us closer to Erd\H{o}s' conjecture we return from the fourth-moment problem to the binary problem, this time using an additive Fourier analysis argument. 

\begin{proof}[Proof of Theorem \ref{positive density theorem}]
Let $\epsilon>0$ and without loss of generality also assume $\epsilon\leqslant \frac{1}{4}$. Take $\eta = q^{-\frac{1}{4}+\epsilon}$. We first show that, for any fixed $a\in (\mathbb{Z}/q\mathbb{Z})^\times$, the number of solutions to $p_1p_2\equiv a (\operatorname{mod} q)$ with $p_1,p_2\in P_\eta$ is at most $ (1+o(1))\frac{q^{\frac{1}{2}+2\epsilon}}{\xi^2 \log^2\eta q}$, with $\xi$ a suitable small constant to be chosen later. Indeed, note that the contribution when one of the primes is less than $q^{\frac{1}{2}}$ is negligible compared to the desired bound, so we may assume that $p_1,p_2 > q^{\frac{1}{2}}$. Now let $x=\eta q$ and take $w^+$ from Proposition \ref{upper bound sieve weights} (we will choose suitable $\xi$ later). By our above observation, and property (i) of Proposition \ref{upper bound sieve weights}, we may upper-bound the number of solutions by $$\sum\limits_{n=1}^{q-1}w^+(n)w^+(an^\ast)$$ where $n^\ast$ denotes the multiplicative inverse modulo $q$. By the additive Fourier inversion formula, this is equal to 

\begin{equation}
\label{ready for Kloosterman}
\frac{1}{q^2}\sum\limits_{r=1}^{q}\sum\limits_{s=1}^q \widehat{w^+}(r)\widehat{w^+}(s)\sum\limits_{n=1}^{q-1}e\left(\frac{rn+san^\ast}{q}\right)
\end{equation}

By property (ii) of Proposition \ref{upper bound sieve weights}, the contribution from the term where $r=s=0$ is at most $(1+o(1))\frac{\eta ^2q}{\xi^2\log^2\eta q}$. The remaining contribution is at most $T_1 +T_2+T_3$, where

\begin{equation}
\begin{cases}
T_1=\frac{1}{q^2}\sum\limits_{r=1}^{q-1}\sum\limits_{s=1}^{q-1}\vert \widehat{w^+}(r) \vert \vert \widehat{w^+}(s) \left\vert \sum\limits_{n=1}^{q-1}e\left(\frac{rn+san^\ast}{q}\right) \right\vert \\
T_2= (1+o(1))\frac{\eta}{q\xi^2\log \eta q} \sum\limits_{r=1}^{q-1}\vert \widehat{w^+}(r) \vert \left\vert\sum\limits_{n=1}^{q-1}e\left(\frac{rn}{q}\right) \right\vert \\
T_3 = (1+o(1))\frac{\eta}{q\xi^2\log \eta q}\sum\limits_{s=1}^{q-1}\vert \widehat{w^+}(s)\vert \left\vert \sum\limits_{n=1}^{q-1}e\left(\frac{san^\ast}{q}\right)\right\vert\nonumber
\end{cases}
\end{equation}

In $T_1$ the inner sum is the Kloosterman sum $\operatorname{Kl_2}(r,sa;q)$ -- see Chapter 11 of \cite{Iw-Ko04} -- which enjoys the Weil bound $$ \operatorname{Kl_2}(r,sa;q)\leqslant 2\sqrt{q}.$$ The other two exponential sums are trivially of size 1, and the sums of the Fourier coefficients of $w$ are precisely of the form estimated in Lemma 6. Hence we may conclude that

\begin{equation}
\begin{cases}
T_1\ll \eta^{4\xi}q^{4\xi+\frac{1}{2}+o(1)}\\
T_2\ll \frac{\eta^{1+2\xi}q^{2\xi+o(1)}}{\xi^2}\\
T_3\ll \frac{\eta^{1+2\xi}q^{2\xi+o(1)}}{\xi^2}\nonumber
\end{cases}
\end{equation}

Substituting $\eta=q^{-\frac{1}{4}+\epsilon}$, a short calculation demonstrates that the main term dominates as $q\rightarrow\infty$ provided that we fix $\xi<\frac{2\epsilon}{3+4\epsilon}$. Picking such a $\xi$, for which a $\gamma = \gamma(\epsilon)$ may be chosen to satisfy the hypotheses of Proposition \ref{upper bound sieve weights}, since $\frac{2\epsilon}{3+4\epsilon}<\frac{1}{2}$, we have concluded that 

\begin{equation}
\label{concentration bound}
\displaystyle\sum\limits_{\substack{p_1,p_2\in P_\eta\\ p_1p_2\equiv a (\operatorname{mod} q)}} 1 \leqslant (1+o(1))\frac{q^{\frac{1}{2}+2\epsilon}}{\xi^2 \log^2 \eta q}
\end{equation} as claimed. 

Now we sum (\ref{concentration bound}) over all $a$ in $P_\eta^{(2)}$. This yields $$ (1+o(1))\frac{\eta^2 q^2}{\log^2\eta q}\leqslant (1+o(1))\lvert P_\eta^{(2)}\rvert \frac{q^{\frac{1}{2}+2\epsilon}}{\xi^2 \log^2\eta q}.$$ Substituting in $\eta=q^{-\frac{1}{4}+\epsilon}$ and rearranging gives $$\lvert P_\eta^{(2)} \rvert \geqslant \xi^2 q(1+o(1))\geqslant \left(\frac{2\epsilon}{3+4\epsilon}\right)^2 q(1+o(1)),$$ by letting $\xi$ tend to $\frac{2\epsilon}{3+4\epsilon}$ from below suitably slowly as $q$ tends to infinity. This proves the theorem. 
\end{proof}

The proof of Theorem 2 is an easy consequence of a standard Fourier analysis argument, namely the use of triple convolutions.
\begin{proof}[Proof of Theorem 2]
We first prove part (i). Let $\epsilon>0$ be a small constant, and take $\eta=q^{-\frac{1}{16}+\epsilon}$, $x=\eta q$ and $w^-$ the weight from Proposition \ref{lower bound sieve weights} (we will choose appropriate $\delta$ and $\xi$ later). Finally define $1_\eta$ to be the indicator function of the set $P_\eta$. 

We proceed by showing that $(w^-\ast 1_\eta\ast 1_\eta)(a) >0$ for all $a\in (\mathbb{Z}/q\mathbb{Z})^\times$.  Indeed by multiplicative Fourier inversion we have the identity

\begin{align}
w^-\ast 1_\eta\ast 1_\eta(a)&=\frac{1}{q-1}\sum\limits_{\chi} \widehat{w^-}(\chi) \widehat{1_\eta}(\chi)^2 \chi(a)\notag\\
&= \frac{(1+o(1))c(\delta)\eta^3q^2}{\xi\log^3\eta q}+\frac{1}{q-1}\sum\limits_{\chi\neq \chi_0} \widehat{w^-}(\chi) \widehat{1_\eta}(\chi)^2\chi(a).
\end{align}
\noindent by Property (iii) of Proposition \ref{lower bound sieve weights}. If the claim $w^-\ast 1_\eta\ast 1_\eta (a) >0$ were false, then we would have $$\frac{1}{q-1}\sum\limits_{\chi\neq \chi_0} \vert\widehat{w^-}(\chi) \widehat{1_\eta}(\chi)^2\vert \geqslant (1+o(1))\frac{c(\delta)\eta^3q^2}{\xi\log^3\eta q}$$ and therefore $$\operatorname{sup}\limits_{\chi\neq\chi_0} \vert \widehat{w^-}(\chi)\rvert \sum\limits_{\chi}\lvert \widehat{1_\eta}(\chi)\rvert^2 \geqslant(1+o(1))\frac{c(\delta)\eta^3 q^3}{\xi\log^3\eta q}.$$ But by Parseval's identity this would imply that 

\begin{equation}
\label{the equation to be contradicted}
\operatorname{sup}\limits_{\chi\neq\chi_0} \vert \widehat{w^-}(\chi)\rvert\geqslant  (1+o(1))c(\delta)\frac{\eta^2 q}{\xi\log ^2 \eta q}.
\end{equation}

\noindent From Lemma \ref{Fourier coefficient prop}, we have $$\operatorname{sup}\limits_{\chi\neq\chi_0} \vert \widehat{w^-}(\chi)\vert\ll (\eta q)^{2\xi+\delta}q^{\frac{1}{2}}\log q.$$ A short calculation shows that this contradicts (\ref{the equation to be contradicted}), provided that 
\begin{equation}
\label{parts to theorem 2 diverge}
\xi < \frac{\frac{1}{4}+\log_q\eta}{1+\log_q\eta}-\frac{\delta}{2}
\end{equation}

For $\eta=q^{-\frac{1}{16}+\epsilon}$, the condition reads $$\xi < \frac{1}{5}+\frac{64\epsilon}{75+80\epsilon} - \frac{\delta}{2},$$ and hence, by picking $\delta>0$ small enough, there exists\footnote{Note that inequality (\ref{parts to theorem 2 diverge}) already implies that $\xi<\frac{1}{4}$, so $\gamma = \frac{1}{4}$ may be used in Proposition 5.} a permissible $\xi$ greater than $\frac{1}{5}$ satisfying the inequality (\ref{parts to theorem 2 diverge}). Picking such a $\xi$, and the contradiction obtained, we conclude that $(w^-\ast 1_\eta\ast 1_\eta)(a) >0$ for all $a\in (\mathbb{Z}/q\mathbb{Z})^\times$. But trivially we then have $\operatorname{max}(w^-,0)\ast 1_\eta\ast 1_\eta (a) >0$ as well. Recalling the statement of Proposition \ref{lower bound sieve weights}, we observe that $\operatorname{max}(w^-,0)$ is an arithmetic function supported on $[\eta q]$, and furthermore only supported on numbers all of whose prime factors are at least $(\eta q)^\xi$. Since $\xi> \frac{1}{5}$, we see that $\operatorname{max}(w^-,0)$ is only supported on numbers with at most 4 primes factors. Part (i) of Theorem \ref{almost prime theorem} is then immediate.\\

To prove part (ii) of the theorem, we take $\eta=q^{-\frac{1}{4}+\epsilon}$ and proceed identically until (\ref{parts to theorem 2 diverge}). By picking $\delta>0$ small enough in terms of $\epsilon$ we may ensure that the upper bound in (\ref{parts to theorem 2 diverge}) is positive, and so there is some $\xi$ satisfying (\ref{parts to theorem 2 diverge}) and some $k(\epsilon)\in \mathbb{N}$ such that $\xi>\frac{1}{k(\epsilon)-1}$. Then $\operatorname{max}(w^-,0)$ is supported on numbers with at most $k(\epsilon)-2$ prime factors, and part (ii) of the theorem immediately follows. 
\end{proof}

\section{Proof of Theorem \ref{exact theorem}}
\label{Proof of Theorem 3}
We adopt the more standard combinatorial notation $A\cdot A = A^{(2)}$. Consider first the following standard combinatorial lemma, which renders precise the notion that being contained in a coset is the only obstruction to a set having reasonable doubling. 

\begin{Lemma} 
\label{Freiman}
Let $(G,\cdot)$ be an abelian group, written multiplicatively, and let $A\subseteq G$. Suppose that $A$ is not contained in any proper coset of $G$. Then either $A\cdot A^{-1}=G$ or $\lvert A\cdot A\rvert \geqslant \frac{3}{2} \lvert A \rvert$.
\end{Lemma} 

\begin{proof}
Originally in \cite{Fr73}, in Russian, but the brevity of the argument allows us to repeat it here. Suppose that $\lvert A\cdot A\rvert < \frac{3}{2} \lvert A \rvert$; we show that $A\cdot A^{-1}$ is closed under multiplication. Indeed, let $w,x,y,z\in A$. The set $$\{a\in A: wa\in zA\}$$ has size greater than $\frac{1}{2}\vert A\rvert$ since $\lvert A\cdot A\rvert < \frac{3}{2} \lvert A \rvert$. Similarly $$ \lvert\{a\in A: xa\in yA\}\rvert > \frac{1}{2}\vert A\rvert.$$ Therefore these two sets intersect, and we have $a,a_z,a_y\in A$ such that $wa=za_z$ and $xa=ya_y$. Hence $$(wx^{-1})(yz^{-1})=waa^{-1}x^{-1}yz^{-1}= za_za_y^{-1}y^{-1}yz^{-1}=a_za_y^{-1}$$

Therefore $A\cdot A^{-1}$ is a subgroup of $G$ -- the other axioms are trivial -- and as $A$ is not contained in any proper coset of $G$ we conclude that $A\cdot A^{-1}$ must be the whole of $G$. 
\end{proof}

We now progress to showing that for $\eta=q^{-\frac{1}{4}+\epsilon}$ the set of primes $P_\eta$ is not contained in any proper coset of $(\mathbb{Z}/q\mathbb{Z})^\times$, allowing the application of the previous lemma. This is equivalent to proving that there is no non-principal character of $(\mathbb{Z}/q\mathbb{Z})^\times$ taking constant values on $P_\eta$. We shall rule out the existence of such pathological characters, in fact for even shorter ranges of primes.

\begin{Lemma}
\label{coset proposition}
Let $\epsilon>0$ and $\eta=q^{-\frac{3}{4}+\epsilon}$. Then there exists a constant $C(\epsilon)$ such that for all primes $q\geqslant C(\epsilon)$ there does not exist a proper subgroup $H\subseteq\left(\mathbb{Z}/q\mathbb{Z}\right)^{\times}$ and $x\in\left(\mathbb{Z}/q\mathbb{Z}\right)^{\times}$ with $P_\eta\subseteq xH$.  
\end{Lemma} 

\noindent Unfortunately our method proves no stronger result. In particular we cannot show that the primes $p$ less than $q$ enjoy any equidistribution in cosets of $(\mathbb{Z}/q\mathbb{Z})^\times$, and so we are restricted to using very general combinatorial arguments such as Lemma \ref{Freiman}, in lieu of Fourier analytic techniques.\\

Let us briefly discuss why proving equidistribution may be a genuinely difficult problem\footnote{Of course we cannot rule out that we have missed some easy argument.}. Proving equidistribution is equivalent to exhibiting some non-trivial upper bound for the character sum $\sum\limits_{p<\eta q}\chi(p)$, for which the standard method is to use zero-free regions of $L$-functions. For example, consider the result (\cite{Iw-Ko04} p. 124) 
\begin{equation}
\label{bound of sum over primes}
\sum\limits_{p<x}\chi(p)\ll_A \sqrt{q}x(\log x)^{-A}
\end{equation}
\noindent for any non-principal character modulo $q$. Regrettably this result is worse than trivial for our applications, as $x\leqslant q$, and the unfortunate $\sqrt{q}$ factor is very stubborn, coming from the $q$ dependence in the $L$-function zero-free region which has resisted improvement for 80 years\footnote{There is some work (some theorems from \cite{Ga06}, for example) showing that for sums over short intervals there are only a few exceptional conductors for which a better cancellation fails to hold, but this does not assist with the consideration of fixed conductor $q$.}. In Chapter 9 of \cite{Mo94} Montgomery gives bounds below which a character is surjective when restricted to primes -- a stronger conclusion than that of Lemma \ref{coset proposition} -- but this is conjectural on a larger zero-free region around $s=1$. Of course, conditional on GRH the left hand side of (\ref{bound of sum over primes}) enjoys almost square-root cancellation, with only logarithmic dependence on $q$. 

The difficulty arises as we are evaluating the sum of a multiplicative function along the primes: much better estimates have long been known to hold for short character sums over shifted primes $p+b$ (for any fixed $b$ coprime to the conductor $q$). In particular, from  \cite{Ka70} we have that 

\begin{equation}
\label{Karatsuba's result}
\sum\limits_{p<q^{\frac{1}{2}+\epsilon}}\chi(p+b)\ll_\epsilon q^{\frac{1}{2}+\epsilon-\delta}
\end{equation}

\noindent for some $\delta>0$ depending only on $\epsilon$. An easy Fourier argument\footnote{Sketch: proceed as in the proof of Theorem \ref{almost prime theorem}, but take all three functions equal to the indicator function of the set of shifted primes $P_\eta + b$. One gets a main term from the principal character, and then the error -- a third moment -- may be bounded using (\ref{Karatsuba's result}) and Parseval's identity. This method may also be used to prove the result of an earlier footnote, that $P_1^{(3)} = (\mathbb{Z}/q\mathbb{Z})^\times$ assuming GRH, using the character sum bound $\sum\limits_{p\leqslant x}\chi(p) \ll x^\frac{1}{2}(\log xq)^2$.} may then be used to establish Theorem \ref{exact theorem} part (ii) with $\eta=q^{-\frac{1}{2}+\epsilon}$, $k=3$, and $P_\eta$ replaced by a shift $P_\eta+b$.\\ 

An alternative approach to ruling out hypothetical conspiracies of characters at primes is to convert such behaviour into a conspiracy over an interval, obtaining a contradiction to the various known estimates for character sums over intervals. For example, using such an approach, one may show that there are prime quadratic non-residues and residues less than $q^{\frac{1}{4}+\epsilon}$. The non-residue case is due to Burgess \cite{Bu57}, who proved a slightly stronger result, and is an immediate application of his famous character sum bound; the residue case is due to Vinogradov and Linnik in \cite{V-L66}, although Pintz gave a much simpler proof in \cite{Pi77}. A natural generalisation of this method for $n^{th}$-power residues, undertaken by Elliott in \cite{El73}, shows that there are primes $p$ less than $q^{\frac{n-1}{4}+\epsilon}$ which are $n^{th}$-power residues modulo $q$ -- the equivalent statement for a non-residue follows immediately from Burgess. Collecting these results together, we see that for $\eta =1$ Lemma \ref{coset proposition} is already known in the case where $\chi$ has order 2, 3 or 4. \\

The limitation of these methods is that they do not seem to be very robust. Weakening the assumptions, we quickly lose control of $\chi$ over a positive density subset of that interval, and the method fails. \\

Having discussed at some length the problems surrounding Lemma \ref{coset proposition}, let us proceed with the proof.

\begin{proof}[Proof of Lemma \ref{coset proposition}]
Let $\eta=q^{-\frac{3}{4}+\epsilon}$ and suppose that $P_\eta$ is contained in some coset of a non-trivial subgroup of $(\mathbb{Z}/q\mathbb{Z})^\times$: equivalently, some Dirichlet character $\chi$ with conductor $q$ (necessarily primitive) is constant on $P_\eta$. We may preclude the case when $\chi$ is the quadratic character by the existing results mentioned above, so without loss of generality $\chi$ is complex. If $\chi(p)\equiv z$ on primes $p<\eta q$, for $z$ some root of unity, then $\chi(n)$ agrees with the function $z^{\Omega(n)}$ for all $n<\eta q$, where $\Omega(n)$ is the number of prime factors of $n$ counted with multiplicity. It will be important that $z$ is bounded uniformly away from $-1$; w.l.o.g. we may assume, by replacing $\chi$ with a suitable power of $\chi$, that $\text{Re}(z)\geqslant \text{Re}(e^{\frac{2\pi i}{3}})=-\frac{1}{2}$. 

We have Burgess' character sum estimate \cite{Bu62}, that for every $r\in\mathbb{N}$

\noindent \begin{equation}
\label{Burgessbound}
\sum\limits_{n\leqslant x}\chi(n)\ll x^{1-\frac{1}{r}}q^{\frac{r+1}{4r^2}} (\log q)^{\frac{1}{r}}.
\end{equation}

A proof of this formulation may be found in Chapter 12 of \cite{Iw-Ko04}. We obtain a contradiction by observing that the equivalent sum with $z^{\Omega(n)}$ does not enjoy the same cancellation. Indeed, Theorem 5.2 in \cite{Te95} shows that for any root of unity $z$ (apart from $z=-1$)

\begin{equation}
\label{tenenbaum}
\sum\limits_{n\leqslant x}z^{\Omega(n)}= x(\log x)^{z-1}\left(\dfrac{\prod\limits_{p}\left(1-\frac{z}{p}\right)^{-1}\left(1-\frac{1}{p}\right)^z}{\Gamma(z)}+O\left(\dfrac{1}{\log x}\right)\right)\nonumber
\end{equation}

\noindent with implied constant independent of $z$. The result is proved using the Selberg-Delange method, though is in essence originally due to Sathe \cite{Sa53}.

Since $\text{Re}(z)\geqslant -\frac{1}{2}$, $z$ lies on a segment of the unit circle on which $\Gamma(z)$ is uniformly bounded, and hence $\frac{1}{\Gamma(z)}$ is bounded away from zero. Further, as $\lvert z \rvert = 1$, the factor $\prod\limits_{p}\left(1-\frac{z}{p}\right)^{-1}\left(1-\frac{1}{p}\right)^z$ is bounded away from zero (immediately seen by taking logarithms). Therefore, for large enough $q$

\begin{equation}
\label{non-cancellation}
\sum\limits_{n\leqslant \eta q}z^{\Omega(n)}\gg \dfrac{\eta q}{\log^{\frac{3}{2}} (\eta q)}
\end{equation}

\noindent for all $z$ with $\lvert z\rvert=1$ and $\text{Re}(z)\geqslant -\frac{1}{2}$, uniformly. For large enough $q$ and $r$, depending on $\epsilon$, (\ref{Burgessbound}) and (\ref{non-cancellation}) immediately combine to obtain a contradiction when $x=\eta q$. This concludes the proof of Lemma \ref{coset proposition}.
\end{proof}

\begin{proof}[Proof of Theorem 3]
\noindent Let $G$ be a finite abelian group, written multiplicatively. Observe that if a set $A\subseteq G$ is not contained in any proper coset then neither is $A^{(k)}$ for all $k\geqslant 1$. Thus Lemma \ref{Freiman} may be applied iteratively. 

We investigate the second case of this lemma in more detail. If $A\cdot A^{-1}=G$ it follows from the Ruzsa triangle inequality (see Chapter 2 of \cite{Ta-Vu06}) that $\lvert A\cdot A \rvert \geqslant  \left(\frac{\lvert G\rvert}{\lvert A\rvert}\right)^{\frac{1}{2}}\lvert A\rvert$ and hence the same bound $\lvert A\cdot A \rvert\geqslant \frac{3}{2} \lvert A \rvert$ applies, provided that $\lvert A \rvert \leqslant \frac{2^2}{3^2}\lvert G \rvert$. In the case when $ \frac{2^2}{3^2}\lvert G \rvert\leqslant \lvert A\rvert \leqslant \frac{1}{2}\lvert G \rvert$, we have the estimate $\lvert A\cdot A \rvert\geqslant \sqrt{2} \lvert A \rvert$. In particular we have $\lvert A\cdot A\vert > \frac{1}{2}\lvert G \rvert$, and hence $A^{(4)}=G$. 

We let $G=(\mathbb{Z}/q\mathbb{Z})^\times$. To prove part (i) of Theorem \ref{exact theorem}, let $\eta = q^{-\frac{1}{16}+\epsilon}$ and apply Lemma \ref{Freiman} iteratively,  starting with $A= P_\eta^{(6)}$. Lemma \ref{coset proposition} ensures that the hypotheses of Lemma \ref{Freiman} are satisfied. Inspecting the proof of Theorem 2 we see that $$G = P_\eta^{(3)}\cup P_\eta^{(4)}\cup P_\eta^{(5)}\cup P_\eta^{(6)}.$$ If $k>k^\prime$ it is clear that the density of $P_\eta^{(k)}$ is at least the density of $P_\eta^{(k^\prime)}$. Therefore we may conclude\footnote{By considering separately the case where $P_\eta^{(3)}$ has small doubling one can restrict to the case $\vert P_\eta^{(6)}\vert\geqslant \frac{3}{11}\lvert G \rvert $, but this is not quite enough to reduce the number of iterations of Lemma \ref{Freiman}.} that $\vert P_\eta^{(6)}\vert \geqslant \frac{1}{4}\lvert G \rvert$. Since $\frac{1}{4}>\frac{2}{3}\times \frac{2}{3}\times  \frac{1}{2}$ we may apply Lemma \ref{Freiman} twice, and then the trivial observation that $\lvert S \rvert > \frac{1}{2}\vert G \vert$ implies $S\cdot S = G$, to conclude that $A^{(8)}= G$. In other words, $P_\eta^{(48)}=(\mathbb{Z}/q\mathbb{Z})^\times$. 

To prove part (ii) of Theorem \ref{exact theorem}, we note that $P_\eta^{(k(\epsilon))}$ has positive density -- either by Theorem 1 or Theorem 2 (ii) -- and so iterating Lemma \ref{Freiman}, starting with $P_\eta^{(k(\epsilon))}$, gives the result. 
\end{proof}

Since Theorem \ref{almost prime theorem} establishes that $P_\eta^{(6)}$ is very large, the proof of Theorem \ref{exact theorem} (i) does not require the full generality of Lemma \ref{coset proposition}. Elliott's bounds \cite{El73} for the least prime quadratic, cubic, and quartic residues suffice in this case. 
\bibliographystyle{alpha}
\bibliography{multschnirelmann}

\begin{thebibliography}{EOS87}

\bibitem[Bur57]{Bu57}
D.~A. Burgess.
\newblock The distribution of quadratic residues and non-residues.
\newblock {\em Mathematika}, 4:106--112, 1957.

\bibitem[Bur62]{Bu62}
D.~A. Burgess.
\newblock On character sums and primitive roots.
\newblock {\em Proc. London Math. Soc. (3)}, 12:179--192, 1962.

\bibitem[Dav00]{Da00}
Harold Davenport.
\newblock {\em Multiplicative number theory}, volume~74 of {\em Graduate Texts
  in Mathematics}.
\newblock Springer-Verlag, New York, third edition, 2000.

\bibitem[Ell71]{El73}
P.~D. T.~A. Elliott.
\newblock The least prime {$k-{\rm th}$}-power residue.
\newblock {\em J. London Math. Soc. (2)}, 3:205--210, 1971.

\bibitem[EOS87]{E-O-S87}
P.~Erd{\H{o}}s, A.~M. Odlyzko, and A.~S{\'a}rk{\"o}zy.
\newblock On the residues of products of prime numbers.
\newblock {\em Period. Math. Hungar.}, 18(3):229--239, 1987.

\bibitem[FI10]{F-I10}
John Friedlander and Henryk Iwaniec.
\newblock {\em Opera de cribro}, volume~57 of {\em American Mathematical
  Society Colloquium Publications}.
\newblock American Mathematical Society, Providence, RI, 2010.

\bibitem[FKS08]{FKS08}
John~B. Friedlander, P{\"a}r Kurlberg, and Igor~E. Shparlinski.
\newblock Products in residue classes.
\newblock {\em Math. Res. Lett.}, 15(6):1133--1147, 2008.

\bibitem[Fre73]{Fr73}
G.~A. Fre{\u\i}man.
\newblock Groups and the inverse problems of additive number theory.
\newblock In {\em Number-theoretic studies in the {M}arkov spectrum and in the
  structural theory of set addition ({R}ussian)}, pages 175--183. Kalinin. Gos.
  Univ., Moscow, 1973.

\bibitem[Gar06]{Ga06}
M.~Z. Garaev.
\newblock Character sums in short intervals and the multiplication table modulo
  a large prime.
\newblock {\em Monatsh. Math.}, 148(2):127--138, 2006.

\bibitem[HBL]{HB-Li15}
D.R. Heath-Brown and Xiannan Li.
\newblock Prime values of $a^2$ + $p^4$.
\newblock \url{http://arxiv.org/abs/1504.00531v1}.

\bibitem[HR74]{Ha-Ri74}
H.~Halberstam and H.-E. Richert.
\newblock {\em Sieve methods}.
\newblock Academic Press [A subsidiary of Harcourt Brace Jovanovich,
  Publishers], London-New York, 1974.
\newblock London Mathematical Society Monographs, No. 4.

\bibitem[IK04]{Iw-Ko04}
Henryk Iwaniec and Emmanuel Kowalski.
\newblock {\em Analytic number theory}, volume~53 of {\em American Mathematical
  Society Colloquium Publications}.
\newblock American Mathematical Society, Providence, RI, 2004.

\bibitem[Kar70]{Ka70}
A.~A. Karacuba.
\newblock Sums of characters with prime numbers.
\newblock {\em Izv. Akad. Nauk SSSR Ser. Mat.}, 34:299--321, 1970.

\bibitem[Mon94]{Mo94}
Hugh~L. Montgomery.
\newblock {\em Ten lectures on the interface between analytic number theory and
  harmonic analysis}, volume~84 of {\em CBMS Regional Conference Series in
  Mathematics}.
\newblock Published for the Conference Board of the Mathematical Sciences,
  Washington, DC; by the American Mathematical Society, Providence, RI, 1994.

\bibitem[Nat96]{Na96}
Melvyn~B. Nathanson.
\newblock {\em Additive number theory}, volume 164 of {\em Graduate Texts in
  Mathematics}.
\newblock Springer-Verlag, New York, 1996.
\newblock The classical bases.

\bibitem[Pin77]{Pi77}
J.~Pintz.
\newblock Elementary methods in the theory of {$L$}-functions. {VI}. {O}n the
  least prime quadratic residue {$({\rm mod}\ \rho )$}.
\newblock {\em Acta Arith.}, 32(2):173--178, 1977.

\bibitem[Sat53]{Sa53}
L.~G. Sathe.
\newblock On a problem of {H}ardy on the distribution of integers having a
  given number of prime factors. {I. - II.}
\newblock {\em J. Indian Math. Soc. (N.S.)}, 17:63--82, 83--141, 1953.

\bibitem[Sch33]{Sch33}
L.~Schnirelmann.
\newblock \"{U}ber additive {E}igenschaften von {Z}ahlen.
\newblock {\em Math. Ann.}, 107(1):649--690, 1933.

\bibitem[Shp12]{Sh12}
Igor~E. Shparlinski.
\newblock Modular hyperbolas.
\newblock {\em Jpn. J. Math.}, 7(2):235--294, 2012.

\bibitem[Shp13]{Sh13}
Igor~E. Shparlinski.
\newblock On products of primes and almost primes in arithmetic progressions.
\newblock {\em Period. Math. Hungar.}, 67(1):55--61, 2013.

\bibitem[Ten95]{Te95}
G{\'e}rald Tenenbaum.
\newblock {\em Introduction to analytic and probabilistic number theory},
  volume~46 of {\em Cambridge Studies in Advanced Mathematics}.
\newblock Cambridge University Press, Cambridge, 1995.
\newblock Translated from the second French edition (1995) by C. B. Thomas.

\bibitem[TV06]{Ta-Vu06}
Terence Tao and Van Vu.
\newblock {\em Additive combinatorics}, volume 105 of {\em Cambridge Studies in
  Advanced Mathematics}.
\newblock Cambridge University Press, Cambridge, 2006.

\bibitem[VL66]{V-L66}
A.~I. Vinogradov and Ju.~V. Linnik.
\newblock Hypoelliptic curves and the least prime quadratic residue.
\newblock {\em Dokl. Akad. Nauk SSSR}, 168:259--261, 1966.

\end{thebibliography}
\end{document}